\documentclass{amsart}
\usepackage{amsmath}
\usepackage{bbm}
\usepackage{latexsym}
\usepackage{xcolor}

 \usepackage[usenames,dvipsnames]{pstricks}
 \usepackage{epsfig}
 \usepackage{pst-grad} 
 \usepackage{pst-plot} 

\newtheorem{thm}{Theorem}
\newtheorem{lemma}[thm]{Lemma}
\newtheorem{cor}[thm]{Corollary}
\newtheorem{prop}[thm]{Proposition}

\newtheorem*{thmi}{Theorem}

\catcode`,\active

\catcode`\,12

\theoremstyle{remark} 
\newtheorem{remark}[]{Remark}
\newtheorem{example}[]{Example}

\newcommand{\Si}{\Sigma}

\newcommand{\ra}{\rightarrow}
\newcommand{\RR}{\mathbb{R}}
\newcommand{\RP}{\mathbb{R}\textrm{P}}
\newcommand{\EE}{\mathbb{E}}
\newcommand{\PP}{\mathbb{P}}
\newcommand{\CC}{\mathbb{C}}

\newcommand{\CP}{\mathbb{C}\textrm{P}}

\newcommand{\e}{\varepsilon}

\title[Intrinsic volumes of sets of singular matrices]{Intrinsic volumes of sets of singular matrices}
\author{Antonio Lerario}
\thanks{The research leading to these results has received funding from the European Community’s Seventh Framework Progamme ([FP7/2007-2013] [FP7/2007-2011]) under grant agreement no [258204]}
\begin{document}

\maketitle
\begin{abstract}Let $\Sigma^{\mu}$ be the set of complex $n\times n$ matrices of Frobenius norm one corank at least $\mu$. We are interested in computing the intrinsic volumes of the two sets:
$$\Si^{\mu}\cap \mathcal{M}(n, \RR)\quad \textrm{and}\quad \Si^{\mu}\cap \textrm{Sym}(n, \RR)$$ 
(they are, respectively, the set of real and real-symmetric $n\times n$ matrices with Frobenius norm one and corank at least $\mu$). 

Denoting by $|X|$ the intrinsic $k$-dimensional volume of a $k$-dimensional semialgebraic set we prove that:
\begin{equation}\label{abs1}\frac{|\Si^{\mu}\cap \mathcal{M}(n, \RR)|}{\left|S^{n^2-\mu^2-1}\right|}=2^{\frac{\mu^2}{2}-1}\Gamma\left(\frac{\mu^2}{2}\right)I_{\mu}\cdot {n\choose \mu}\frac{\prod_{j=1}^{n-\mu}\Gamma(j/2+1)\Gamma(j/2+\mu)}{\prod_{j=1}^n \Gamma(j/2)\Gamma(j/2+1)}.\end{equation}
where:
$$I_\mu=\frac{\pi^{\mu/2}}{2^{\mu^2/2}}\int_{B(0,1)\cap \RR^{\mu}_+}\prod_{1\leq i<j\leq \mu}|\sigma_i^2-\sigma_j^2|d\sigma_1\cdots d\sigma_{\mu}.$$

Similarly, for the case of symmetric matrices we obtain:
\begin{equation}\label{abs2}\frac{|\Si^{\mu}\cap \textrm{Sym}(n, \RR)|}{\left|S^{\frac{n(n+1)}{2}-\frac{\mu(\mu+1)}{2}-1}\right|}=\frac{2^{\frac{\mu(\mu+1)}{4}-1}\Gamma\left(\frac{\mu(\mu+1)}{4}\right)\Gamma\left(\frac{\mu+1}{2}\right)}{\sqrt{\pi}}I_{1, \mu}\cdot {n\choose \mu}\frac{\prod_{i=0}^{m-1}\frac{\Gamma(i+\mu+3/2)}{\Gamma(i+3/2)}}{\prod_{k=n-\mu+1}^{n}\Gamma(1+k/2)}\end{equation}
where:
$$n-\mu=2m+1\quad \textrm{and}\quad I_{1, \mu}=2^{-\mu}\int_{B(0,1)\cap \RR^\mu}\prod_{i<j}|\lambda_i-\lambda_j|d\lambda_1\cdots d\lambda_{\mu}$$

We perform an asymptotic analysis of \eqref{abs1} and \eqref{abs2} and obtain:
$$\frac{|\Si^{\mu}\cap \mathcal{M}(n, \RR)|}{\left|S^{n^2-\mu^2-1}\right|}=\Theta\left( n^{\frac{\mu^2}{2}}\right)\quad \textrm{and}\quad\frac{|\Si^{\mu}\cap \textrm{Sym}(n, \RR)|}{\left |S^{\frac{n(n+1)}{2}-\frac{\mu(\mu+1)}{2}-1}\right|}=\Theta\left(n^{\frac{\mu(\mu+1)}{4}}\right)$$
We discuss several examples and applications.
\end{abstract}
\section{Introduction}

In this paper we consider the problem of computing the intrinsic volume of the sets of \emph{real} matrices with Frobenius norm one and given corank. Besides the pure mathematical interest, this problem is relevant for numerical analysis (because of its connection to the condition number of a linear system \cite{BCSS, Demmel, EdelmanKostlan95, EdelmanKostlanShub}),  and is recently become of interest for an emerging field, \emph{random real algebraic geometry}, as it is discussed in \cite{Lerario2012, gap} (see also examples \ref{randomsym} and \ref{randombetti} below).
\subsection{The complex case}
To start with, one defines the algebraic set:
$$\Si^{\mu}=\{\textrm{$Q\in \mathcal{M}(n, \CC)$ such that $\|Q\|^2=1$ and $\dim\ker(Q)=\mu$}\}$$
(i.e. the set of complex $n\times n$ matrices of \emph{Frobenius} norm one and kernel of dimension $\mu$).
The \emph{real} dimension of $\mathcal{M}(n, \CC)$ is $2n^2$ and $\Si^\mu$ is an algebraic subset of the unit sphere $S^{2n^2-1}$ of real codimension\footnote{All dimensions and codimension are going to be the real ones.} $2\mu^2$; the set of its smooth points equals $\Si^{\mu}\backslash \Si^{\mu+1}$ and it is a $(2n^2-2\mu^2-1)$-dimensional manifold; we denote by $|\Si^\mu|$ its intrinsic volume (a Riemannian metric is induced on it from the one of the sphere)\footnote{From now on, for a given semialgebraic subset $X$ of the sphere, the Euclidean space or the projective space, of dimension $k$ we denote by $|X|$ its intrinsic $k$-dimensional volume, by restricitng the volume form of the ambient space to the set of smooth points of $X$ (if $X$ is singular its volume is the volume of the set of its smooth points).}
. 

Similarly one can consider the algebraic set:
$$\Si^{\mu}\cap\textrm{Sym}(n, \CC)=\{\textrm{$Q\in \textrm{Sym}(n, \CC)$ such that $\|Q\|^2=1$ and $\dim\ker(Q)=\mu$}\},$$
(i.e. the set of complex $n\times n$ symmetric matrices of norm one and kernel of dimension $\mu$).
The \emph{real} dimension of $\textrm{Sym}(n, \CC)$ is $n(n+1)$ and $\Si^\mu\cap \textrm{Sym}(n, \CC)$ is an algebraic subset of the unit sphere $S^{n(n+1)-1}$ of codimension $\mu(\mu+1)$; the set of its smooth points is a smooth $(n(n+1)-\mu(\mu+1)-1)$-dimensional manifold and we denote by $|\Si^\mu \cap \textrm{Sym}(n, \CC)|$ its intrinsic volume (again a Riemannian metric is induced on it from the one of the sphere).

The intrinsic volumes of these sets are easily computed using the integral geometry formula  once one knows the degree of their projectivizations (see Section \ref{sec:complex} below) and are given by:
\begin{equation}\label{eqcplx}\frac{|\Si^{\mu}|}{|S^{2n^2-2\mu^2-1}|}=\prod_{k=0}^{\mu-1}\frac{(n+k)!k!}{(n-\mu+k)!(\mu+k)!}\quad\textrm{and}\quad \frac{|\Si^{\mu}\cap \textrm{Sym}(n, \CC)|}{|S^{n(n+1)-\mu(\mu+1)-1}|}=\prod_{k=0}^{\mu-1}\frac{{{n+k}\choose{\mu-k}}}{{{2k+1}\choose{k}}}\end{equation}
\begin{example}[Complex matrices with determinant zero]
The set $\Sigma^1$ of complex matrices with Frobenius norm one and determinant zero has volume:
$$|\Si^{1}|=|S^{n^2-3}|n=\frac{2n\pi^{n^2-1}}{\Gamma(n^2-1)},$$
and similarly:
$$|\Si^{1}\cap\textrm{Sym}(n, \CC)|=\frac{2n\pi^{\frac{n^2+n-2}{2}}}{\Gamma\left(\frac{n^2+n-2}{2}\right)}.$$
These formulas follows from equations \eqref{eqcplx} above, setting $\mu=1$.
\end{example}
\subsection{The real case}
Moving to our problem, we start by considering the set: $$\Sigma^\mu\cap \mathcal{M}(n, \RR)=\{\textrm{$n\times n$ real matrices of Frobenius norm one and corank $\mu$}\}.$$
This is an algebraic subset of the sphere $S^{n^2-1}$ in $\mathcal{M}(n, \RR)$ and has codimension $\mu^2$. We endow the sets of its smooth points with a Rimenannian metric from the one of the sphere and denote by $|\Sigma^\mu\cap \mathcal{M}(n, \RR)|$ its intrinsic volume (the set of singular points of $\Sigma^{\mu}\cap \mathcal{M}(n, \RR)$ equals $\Si^{\mu+1}\cap \mathcal{M}(n, \RR)$).

Similarly one considers:
$$\Sigma^{\mu}\cap \textrm{Sym}(n,\RR)=\{\textrm{$n\times n$ real symmetric matrices of Frobenius norm one and corank $\mu$}\}.$$
This is an algebraic subset of the sphere $S^{\frac{n(n+1)}{2}-1}$ in $\textrm{Sym}(n, \RR)$ of codimension $\mu(\mu+1)/2$ and we endow the sets of its smooth points with the induced Riemannian metric; we denote by $|\Sigma^{\mu}\cap\textrm{Sym}(n, \RR)|$ its intrinsic volume (again the set of singular points coincides with the set of matrices with corank one more).

The degree of these sets is known, but the integral geometry approach provides only upper bounds for their intrinsic volumes. Using techniques from Random Matrix Theory we obtain the following results.
\begin{thmi}[The case of $n\times n$ real matrices]
\begin{equation}\label{formula}\frac{|\Si^{\mu}\cap \mathcal{M}(n, \RR)|}{\left|S^{n^2-\mu^2-1}\right|}=2^{\frac{\mu^2}{2}-1}\Gamma\left(\frac{\mu^2}{2}\right)I_{\mu}\cdot {n\choose \mu}\frac{\prod_{j=1}^{n-\mu}\Gamma(j/2+1)\Gamma(j/2+\mu)}{\prod_{j=1}^n \Gamma(j/2)\Gamma(j/2+1)}.\end{equation}
where:
$$I_\mu=\frac{\pi^{\mu/2}}{2^{\mu^2/2}}\int_{B(0,1)\cap \RR^{\mu}_+}\prod_{1\leq i<j\leq \mu}|\sigma_i^2-\sigma_j^2|d\sigma_1\cdots d\sigma_{\mu}.$$
\end{thmi}

\begin{example}[Real matrices with zero determinant]The volume of the set of real $n\times n$ matrices with Frobenius norm one and determinant zero was computed in \cite{EdelmanKostlan95, EdelmanKostlanShub}. Here we recover this result by letting $\mu=1$ in the above formula:
$$ \frac{|\Si^{1}\cap \mathcal{M}(n, \RR)|}{\left|S^{n^2-2}\right|}=\frac{\sqrt{\pi}\Gamma\left(\frac{n+1}{2}\right)}{\Gamma\left(\frac{n}{2}\right)}.$$
Using the asymptotic $\Gamma(z+a)/\Gamma(z+b)\sim z^{a-b}$ (with $z=n/2, a=1/2$ and $b=0$), we obtain the asymptotic formula:
$$ \frac{|\Si^{1}\cap \mathcal{M}(n, \RR)|}{\left|S^{n^2-2}\right|}\sim \sqrt{\frac{\pi}{2}}n^{1/2}.$$
\end{example}

For the case of symmetric matrices we obtain the following analogous result; we state it  here only for the case $n-\mu$ is odd, as the general case involves more explicit quantities from Random Matrix theory (see Proposition \ref{limsym}).
\begin{thmi}[The case of $n\times n$ real symmetric matrices]$$\frac{|\Si^{\mu}\cap \emph{\textrm{Sym}}(n, \RR)|}{\left|S^{\frac{n(n+1)}{2}-\frac{\mu(\mu+1)}{2}-1}\right|}=\frac{2^{\frac{\mu(\mu+1)}{4}-1}\Gamma\left(\frac{\mu(\mu+1)}{4}\right)\Gamma\left(\frac{\mu+1}{2}\right)}{\sqrt{\pi}}I_{1, \mu}\cdot {n\choose \mu}\frac{\prod_{i=0}^{m-1}\frac{\Gamma(i+\mu+3/2)}{\Gamma(i+3/2)}}{\prod_{k=n-\mu+1}^{n}\Gamma(1+k/2)}$$
where:
$$n-\mu=2m+1\quad \textrm{and}\quad I_{1, \mu}=2^{-\mu}\int_{B(0,1)\cap \RR^\mu}\prod_{i<j}|\lambda_i-\lambda_j|d\lambda_1\cdots d\lambda_{\mu}$$
\end{thmi}

\begin{example}[Degenerate symmetric and hermitian matrices]A similar computation can be performed for the set of hermitian and quaternionic hermitian matrices. They correspond to the classical Gaussian $\beta$-ensembles ($\beta=1$ is the GOE, $\beta=2$ the GUE and $\beta=4$ the GSE). 

In the case $\mu=1$, the volume of singular symmetric matrices ($\beta=1$) with Frobenius norm one (or more generally hermitian or quaternionic hermitian) has been previously computed in \cite{gap}, regardless the parity of $n$:
\begin{equation}\label{volumeasympt}|\Sigma_{\beta, n}|\sim |S^{N_\beta-2}|\cdot \frac{2}{\sqrt{\pi}}n^{1/2},\end{equation}
where $N_\beta=n+\frac{1}{2}n(n-1)\beta.$ 
For the specific case $\beta=1$ (symmetric matrices), combining the above formula into the previous theorem (notice that $I_{1,1}=1$), we obtain:
$$ \frac{|\Si^{1}\cap \textrm{Sym}(n, \RR)|}{\left|S^{N-2}\right|}=n\sqrt{\frac{2}{\pi}}\frac{\Gamma\left(\frac{n+1}{2}\right)}{\Gamma\left(\frac{n+2}{2}\right)}$$
(the asymptotic \eqref{volumeasympt} follows by plugging in the asymptotic $\Gamma(z+a)/\Gamma(z+b)\sim z^{a-b}$ for $z=n/2$, $a=1/2$ and $b=1$ into the above formula).
\end{example}

The next two examples comes from random real algebraic geometry, and concern average geometric properties of objects defined by random equations.
\begin{example}[The average number of singular points of a real determinantal surface]\label{randomsym}

A real self-adjoint determinantal surface is a surface $S$ in $\RP^3$ defined by the equation:
$$S=\{[x_0:x_1:x_2:x_3]\in \RP^3\,|\, \det(x_0Q_0+x_1Q_1+x_2Q_2+x_3Q_3)=0\}$$
where $Q_0, \ldots, Q_3\in\textrm{Sym}(n, \RR)$ (the degree of $S$ is $n$).

Such a description for the surface $S$ is called a \emph{determinantal representation}. These objects are of interest in Mathematical Physics \cite{ALP}, semidefinte programming \cite{Vinnikov} and real algebraic geometry \cite{DeItKh, complexity, gap}.

Alternatively one can define the surface $S$ as follows. Consider the vector space $W=\textrm{span}\{Q_0, Q_1, Q_2, Q_3\}\subset \textrm{Sym}(n, \RR)$, then:
$$S=P(W\cap \Sigma^1)$$
(here $P(X)$ denotes the projectivization of a set $X$).
For the generic choice of $Q_0, \ldots, Q_3$ the set of singular points of $S$ coincides with:
$$\textrm{Sing}(S)=P(W\cap \Sigma^2).$$
Surfaces arising in this way are very special; for example they have unavoidable singularities (see \cite{complexity}). The complex part of $S$ is defined by:
$$S_{\CC}=P(W_{\CC}\cap \Sigma^1)\subset \CP^3$$
(in other words one consider the set of solutions of the above determinantal equation over the complex numbers). For the generic choice of $W$, the number of singular points of the complex surface $S_\CC$ is:
$$\frac{n(n-1)(n+1)}{6}=\textrm{deg}(\Sigma^2).$$
One can ask how many of these points are \emph{real} on average. For this we simply have to apply the integral geometry formula and we obtain:
$$\EE \textrm{Card}(\textrm{Sing}(S))=\frac{|\Si^{2}\cap \textrm{Sym}(n, \RR)|}{\left|S^{\frac{n(n+1)}{2}-4}\right|}.$$
Assuming $n$ is odd we can compute all the quantities involved for the computations of the previous term, using the above Theorem for $\mu=2$:
$$\frac{2^{\frac{\mu(\mu+1)}{4}-1}\Gamma\left(\frac{\mu(\mu+1)}{4}\right)\Gamma\left(\frac{\mu+1}{2}\right)}{\sqrt{\pi}}=\frac{1}{2}\sqrt{\frac{\pi}{2}}\quad\textrm{and}\quad I_{1, 2}=\frac{1}{4}\int_{0}^1\int_{0}^{2\pi}r |\cos(\theta)-\sin(\theta)|dr d\theta=\frac{\sqrt{2}}{2}.$$
Also we have:
$$\frac{\prod_{i=0}^{m-1}\frac{\Gamma(i+\mu+3/2)}{\Gamma(i+3/2)}}{\prod_{k=n-\mu+1}^{n}\Gamma(1+k/2)}=\frac{8\Gamma\left(\frac{n}{2}\right)}{3\pi\Gamma\left(\frac{n+1}{2}\right)}.$$
Plugging all this into the statement of the aboveTheorem, we obtain:
$$\frac{|\Si^{2}\cap \textrm{Sym}(n, \RR)|}{\left|S^{N-4}\right|}=\frac{n(n-1)\Gamma\left(\frac{n}{2}\right)}{3\sqrt{\pi}\Gamma\left(\frac{n+1}{2}\right)}\sim \sqrt{\frac{2}{9\pi}}\cdot n^{\frac{3}{2}}.$$
Thus even if $S_\CC$ has $\Theta(n^3)$ many singular points, on average only $\Theta(n^{3/2})$ are real!

\end{example}

\begin{example}[Betti numbers of an intersection of two random quadrics]\label{randombetti}This example is taken from \cite{Lerario2012, gap}. Let us consider the problem of computing the average ``topology'' of the algebraic set:
$$X=\{[x]\in \RP^n\, |\, q_1(x)=q_2(x)=0\}$$
where $q_1$ an $q_2$ are random quadratic forms ($X$ is an intersection of two random quadrics; generically it is either a smooth manifold of dimension $n-2$ or it is empty). 

The distribution of probability on $q_1$ and $q_2$ is chosen in such a way that the symmetric matrices obtained by the equation $q_i(x)=\langle x, Q_ix\rangle$ are in the Gaussian Orthogonal Ensemble and they are independent (this is the so called \emph{Kostlan} distribution, see \cite{Bu, EdelmanKostlan95}). 

For example, if $n=2$, $X$ is the set of solutions in the projective plane of a system of two random independent quadratic equations, homogeneous in three variables, (an intersection of conics) and on average consists of $2$ many points (see \cite{ShubSmale}).

Denoting as above $W=\textrm{span}\{Q_1, Q_2\}\subset \textrm{Sym}(n, \RR)$, the sum of the Betti numbers of $X$ is given by \cite{Lerariotwo, complexity}:
\begin{equation}\label{randomtwo}b(X)=3n-4\max_{Q\in W}\textrm{i}^+(Q)+\frac{1}{2}\textrm{Card}(W\cap \Si^1)+O(1).\end{equation}

In the above formula $\textrm{i}^+(Q)$ denotes the number of positive eigenvalues of the symmetric matrix $Q$. If $Q$ is a GOE matrix, then $\EE\max_{Q\in W}\textrm{i}^+(Q)=\frac{n}{2}+O(n^{\alpha})$ for every $0<\alpha<1$ (see \cite[Proposition 19]{gap}). 

The expectation of $\frac{1}{2}\textrm{Card}(W\cap \Si^1)$ is computed using the above theorem combined with the integral geometry formula, and equals: 
$$\EE\frac{1}{2}\textrm{Card}(W\cap \Si^1)= \frac{2}{\sqrt{\pi}}n^{1/2}+O(1).$$
Plugging these into \eqref{randomtwo} one obtains:
$$\EE b(X)=n+\frac{2}{\sqrt{\pi}}n^{1/2}+O(n^{\alpha})\quad \textrm{for every $0<\alpha<1$}$$\end{example}

\subsection{Asymptotics for the real case}
It is interesting to ask what is the order of growth in $n$ of the previous results. Before going to the real case, we discuss the complex one.
A simple analysis of equation \eqref{eqcplx} provides in fact:
\begin{equation}\label{eqcplxas}\frac{|\Si^{\mu}|}{|S^{2n^2-2\mu^2-1}|}=\Theta\left(n^{\mu^2}\right)\quad\textrm{and}\quad \frac{|\Si^{\mu}\cap \textrm{Sym}(n, \CC)|}{|S^{n(n+1)-\mu(\mu+1)-1}|}=\Theta\left(n^{\frac{\mu(\mu+1)}{2}}\right)\end{equation}
For the real case we prove the follwing.

\begin{thmi}[Asymptotic for the case of real and real-symmetric matrices]
\begin{equation}\label{asyreal}\frac{|\Si^{\mu}\cap \mathcal{M}(n, \RR)|}{\left|S^{n^2-\mu^2-1}\right|}=\Theta\left( n^{\frac{\mu^2}{2}}\right)\quad \textrm{and}\quad\frac{|\Si^{\mu}\cap \emph{\textrm{Sym}}(n, \RR)|}{\left|S^{\frac{n(n+1)}{2}-\frac{\mu(\mu+1)}{2}-1}\right|}=\Theta\left(n^{\frac{\mu(\mu+1)}{4}}\right).\end{equation}
\end{thmi}

Notice that the exponents appearing in \eqref{eqcplxas} and \eqref{asyreal} are one-half the \emph{codimension} in the sphere of the algebraic sets we are considering. Moreover in \eqref{asyreal} this exponent is one-half the corresponding one for the complex case (this result can be naively interpreted as saying that the normalized volumes of the real parts grow as the ``square root'' of the complex ones).

\subsection{Structure of the paper}In Section \ref{sec:complex} we discuss the complex case. Section \ref{sec:tubes} is devoted to a preliminary reduction of the problem, using tubes and generalizing a theorem of Eckart and Young. Section \ref{sec:RMT} contains the main computations: the theorem for the real $n\times n$ case is Theorem \ref{volumereal} and for the symmetric case is Theorem \ref{volsym}. Section \ref{sec:asy} contains the asymptotic results: Theorem \ref{thm:realas} for the general $n\times n$ and Theorem \ref{thm:symas} for the symmetric case.

\section{Complex case and the Integral Geometry Formula}\label{sec:complex}
Let us recall the definition of $\Si^{\mu}$:
$$\Si^{\mu}=\{\textrm{$Q\in \mathcal{M}(n, \CC)$ such that $\|Q\|^2=1$ and $\dim\ker(Q)=\mu$}\}.$$
We denote by $2n^2$ the \emph{real} dimension of $\mathcal{M}(n, \CC)$; then $\Si^\mu$ is an algebraic subset of the unit sphere $S^{2n^2-1}$ of codimension\footnote{All dimensions and codimension are the \emph{real} ones.} $2\mu^2$ and we are interested in computing the intrinsic volume of the set of its smooth points.

Over the complex number this computation reduces to the calculation of its degree, via the integral geometry, as we show now. First let us define $\deg(\Si^{\mu})$ as the number of points in the intersection of $P(\Si^{\mu})$ (the projectivization of $\Si^{\mu}$ in $ \CP^{n^2-1}$) with a typical $\CP^{2}$. Let us also recall the integral geometry formula in complex projective spaces \cite{Howard}:
\begin{equation}\label{IGF}\frac{1}{|U(N)|}\int_{U(N)}\frac{|A\cap g B|}{|\CP^{N-1-a-b}|}dg=\frac{|A|}{|\CP^{a}|}\frac{|B|}{|\CP^{b}|}\end{equation}
where $A$ and $B$ are complex submanifolds of $\CP^{N-1}$ of dimensions respectively $a$ and $b$. Applying this formula with $N=n^2$, $A=P(\Si^{\mu})$ and $B=\CP^2$ we obtain:
$$\deg(\Si^{\mu})=\frac{|P(\Si^{\mu})|}{|\CP^{n^2-1-\mu^2}|}=\frac{|\Si^{\mu}|}{|S^{2n^2-2\mu^2-1}|}.$$
In the first equality we have used that the integrand on the r.h.s. of \eqref{IGF} equals $\deg(\Si^{\mu})$ on a full measure set; the second inequality follows from the fact that the volume form on $\CP^{n^2-1}$ is induced by the quotient map $q:S^{2n^2-1}\to \CP^{n^2-1}$; in other words $q$ is a riemannian submersion with fibers $S^1$ and for every $X\subset \CP^{n^2-1}$ measurable we have $|X|=\frac{|q^{-1}(X)|}{2\pi}$ (in particular $|S^{2n^2-1}|=2\pi|\CP^{n^2-1}|$). The degree of $\Si^{\mu}$ is well known and provides:
\begin{equation}\label{comp}\frac{|\Si^{\mu}|}{|S^{2n^2-2\mu^2-1}|}=\deg(\Si^{\mu})=\prod_{k=0}^{\mu-1}\frac{(n+k)!k!}{(n-\mu+k)!(\mu+k)!}=\Theta\left(n^{\mu^2}\right).\end{equation}

Replace now $\mathcal{M}(n, \CC)$ with the space $\textrm{Sym}(n, \CC)$; using the computation of the corresponding degree given in \cite{HarrisTu} we obtain:
\begin{equation}\label{compsym}\frac{|\Si^{\mu}\cap \textrm{Sym}(n, \CC)|}{|S^{n(n+1)-\mu(\mu+1)-1}|}=\deg(\Si^{\mu}\cap \textrm{Sym}(n, \CC))=\prod_{k=0}^{\mu-1}\frac{{{n+k}\choose{\mu-k}}}{{{2k+1}\choose{k}}}=\Theta\left(n^{\frac{\mu(\mu+1)}{2}}\right).\end{equation}


Notice that we can set $V_{\CC}$ to be the vector space $\mathcal{M}(n, \CC)$ or $\textrm{Sym}(n, \CC)$ with real dimension $N$ respectively equal to $2n^2, n(n+1)$ and denoting by $c_{\mu}$ the codimension of $\Si^{\mu}\cap \mathcal{M}_{\CC}$ in $\mathcal{M}_{\CC}$, we can write the two above equations as:
$$\frac{|\Si^\mu\cap V_{\CC}|}{|S^{N-c_{\mu}-1}|}=\Theta\left( n^{c_{\mu}/2}\right).$$
\section{Tubes and Eckart-Young theorems}\label{sec:tubes}
\begin{lemma}Let $X\subset S^{N-1}$ be a smooth submanifold with finite volume and let $c=\textrm{codim}_{S^{N-1}}(X)$. Then:
$$|X|=\lim_{\e\to 0}\frac{|\mathcal{U}_{S^{N-1}}(X, \e)|}{|S^{c-1}|\e^c}.$$
\begin{proof}
Let us set $v(\e)=|\mathcal{U}_{S^{N-1}}(X, \e)|$; we recall Weyl's tube formula \cite{Bu,tubes} for submanifolds of the sphere:
\begin{equation}\label{tube}v(\e)=\sum_{0\leq l\leq N-1-c}K_{c+l}(X)J_{N-1, c+l}(\e)\end{equation}
where the functions $K_{c+l}(X)$ are metric invariants of $X$ such that:
$$K_{c}(X)=|X||S^{c-1}|\quad \textrm{and}\quad J_{N-1, c+l}(\e)=\int_{0}^\e(\sin t)^{c+l-1}(\cos t)^{N-1-c-l}dt.$$
Notice that $J_{N-1, c+l}(0)=0$ and more generally also $J^{(r)}_{N-1, c+l}(0)=0$ if $r<c+l$. In fact we have $J'_{N-1, c+l}(\e)=(\sin \e)^{c+l-1}(\cos \e)^{N-1-c-l}$ and replacing $\sin \e$ an $\cos \e$ with their taylor polynomial at zero we obtain $J^{(r)}_{N-1, c+l}(0)=\frac{\partial}{\partial \e^{r-1}}(\e^{c+l-1})|_{\e=0}=0$ (if $r<c+l$). Thus:
\begin{equation}\label{J}J_{N-1, c+l}(\e)=\e^{c+l}\underbrace{h_{N-1, c+l}(\e)}_{\textrm{bounded}}.\end{equation}
This provides:
\begin{align*}\lim_{\e\to0}\frac{v(\e)}{\e^c}&=\sum_{0\leq l\leq N-1-c}K_{c+l}(X)\lim_{\e\to0}\frac{J_{N-1, c+l}(\e)}{\e^c}\\
&=K_c(X)\lim_{\e\to 0}\frac{J_{N-1, c}(\e)}{\e^c}=\lim_{\e\to 0}\frac{J'_{N-1, c}(\e)}{\e^{c-1}}\\
&=K_{c}(X)\lim_{\e\to 0}\frac{(\sin \e)^{c+l-1}(\cos \e)^{N-1-c-l}}{\e^{c-1}}\\
&=K_c(X)=|X||S^{c-1}|.
\end{align*}
\end{proof}

\end{lemma}

\begin{remark}
Notice in particular that combining equations \eqref{tube} and \eqref{J} we obtain:
\begin{equation}\label{volas}v(\e)=\e^c|X||S^{c-1}|+O(\e^{c+1}).\end{equation}
\end{remark}
\begin{figure}
\scalebox{0.5} 
{
\begin{pspicture}(0,-1.0492039)(12.001895,6.07)
\psline[linewidth=0.04cm](4.98,-0.97)(4.98,6.05)
\psarc[linewidth=0.04](5.2,-0.85){5.2}{358.02505}{92.29061}
\psline[linewidth=0.04cm](4.98,-0.97)(7.54,3.77)
\psline[linewidth=0.04cm](7.54,3.75)(4.98,3.77)
\usefont{T1}{ptm}{m}{n}
\rput(5.7914553,5.675){$Z^\mu\cap V$}
\usefont{T1}{ptm}{m}{n}
\rput(11.141455,-0.225){$S^{N-1}$}
\usefont{T1}{ptm}{m}{n}
\rput(6.111455,4.535){$y$}
\usefont{T1}{ptm}{m}{n}
\rput(6.081455,3.495){$\sin y$}
\psdots[dotsize=0.12](7.56,3.75)
\usefont{T1}{ptm}{m}{n}
\rput(7.8114552,3.955){$X$}
\end{pspicture} 
}
\caption{For small enough $y$ we have $y=d_{S^{N-1}}(X, \Sigma^\mu\cap V)\leq \sin y+(\sin y)^2$}\label{figsen}

\end{figure}
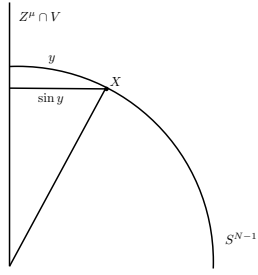
\begin{lemma}\label{vol}
Let $Z^{\mu}$ be the set of real $n\times n$ matrices of corank $\mu$ and denote by $V$ respectively the vector space $\mathcal{M}(n, \RR)$ or $\emph{\textrm{Sym}}(n, \RR)$. Let us denote the \emph{real} dimension of $V$ by $N$ (thus $N$ equals respectively $n^2$ and  $n(n+1)/2$). The unit sphere in $V$ is denoted by $S^{N-1}$ and the codimension of $\Si^\mu\cap V$ in $V$ by $c$. We have:
\begin{equation}\label{tubetobig}|\Si^{\mu}\cap V|=\lim_{\e\to 0}\frac{|\mathcal{U}_{V}(Z^{\mu}, \e)\cap S^{N-1}|}{|S^{c-1}|\e^c}.\end{equation}
\end{lemma}
\begin{proof}

Consider the two functions $v(\e)=|\mathcal{U}_{S^{N-1}}(\Si^{\mu}\cap V,\e)|$ and $\hat{v}(\e)=|\mathcal{U}_{V}(Z^{\mu}, \e)\cap S^{N-1}|$ (where $\mathcal{U}_{V}(Z^{\mu}, \e)$ is an $\e$-tube of $Z^{\mu}$ in $V$). We will prove that these functions have the same order at zero, with the same leading constant $|X||S^{c-1}|$. 

First notice that if $d_{S^{N-1}}(X, \Sigma^{\mu}\cap V)\leq \e$ then also $d_{V}(X, Z^{\mu})\leq \e$; hence:
\begin{equation}\label{vol1}\mathcal{U}_{S^{N-1}}(\Si^{\mu}\cap V, \e)\subset \mathcal{U}_{V}(Z^\mu, \e)\cap S^{N-1}.\end{equation}

Assume now that $d_{V}(X, Z^\mu)\leq \e$ for $X\in S^{N-1}.$ Then the geodesic joining $X$ to $\Sigma^{\mu}\cap V$ is an ``arc'' on the sphere of length $y=d_{S^{N-1}}(X, \Sigma^{\mu}\cap V)$. Since $\Sigma^{\mu}\cap V$ is a homogeneous cone, then $d_{V}(X, Z^\mu)=\sin y$ (see Figure \ref{figsen}) and for $y$ small enough: 
$$y=\sin y+\frac{y^3}{6}+O(y^4)\leq \sin y +y^3\leq \sin y+(\sin y)^2\leq \e+\e^2.$$ In particular we get the inclusion:
\begin{equation}\label{vol2}\mathcal{U}_{V}(Z^{\mu}, \e)\cap S^{N-1}\subset \mathcal{U}_{S^{N-1}}(\Sigma^{\mu}\cap V, \e+\e^2).\end{equation}

Combining (\ref{vol1}) and (\ref{vol2}) we obtain the chain of inequalities: 
\begin{equation}\label{limit}v(\e) \leq \hat{v}(\e) \leq v(\e+\e^2).\end{equation} We can now use equation \eqref{volas}, which provides:
$$v(\e+\e^2)=\e^c|X||S^{c-1}|+\e^{2c}|X||S^{c-1}|+O(\e^{c+1})=\e^c|X||S^{c-1}|+O(\e^{c+1}),$$
and combining this into \eqref{limit} we finally obtain:
$$\lim_{\e\to 0}\frac{v(\e)}{\e^c}=\lim_{\e\to 0}\frac{\hat{v}(\e)}{\e^c}.$$\end{proof}

For a matrix $Q\in \mathcal{M}(n, \RR)$ let us consider its \emph{singular values} $\sigma_1, \ldots, \sigma_n$, i.e. the eigenvalues (not ordered by their magnitude) of $(QQ^T)^{1/2}.$
The Eckart-Young theorem states that the distance, in the Frobenius norm, between a nonsingular matrix $Q$ and the set $Z$ equals the least singular value of $Q$ (the singular value of smallest magnitude). More generally we have the following.

\begin{prop}\label{cone}With the same notation as in Lemma \ref{vol}, we have:
$$\mathcal{U}_{V}(Z^\mu, \e)\cap S^{N-1}=\{Q\in S^{N-1}\,|\, \sigma_{i_1}(Q)^2+ \cdots+ \sigma_{i_{\mu}}(Q)^2\leq \e^2\textrm{ for some $1\leq i_{1}, \ldots, i_{\mu}\leq n$}\}$$

\end{prop}
\begin{proof}
For a matrix $Q\in V$ consider the function $d_{Q,\mu}:Z^\mu\to \RR$ defined by:
$$d_{Q, \mu}(X)=d_{V}(X,Q).$$
Notice that we are not considering the distance between $Q$ and $X$ in $\mathcal{M}(n, \mathbb{R})$, but rather their distance in $V$ (it could possibly be bigger). Thus if we want to find critical points of $d_{\mu, Q}$ for $V\neq \mathcal{M}(n, \RR)$ we cannot apply Eckart-Young theorem in its classical version \cite{BCSS}. We use the generalization of this theorem proved for the space of symmetric matrices in \cite{H, gap}.

We obtain critical points of $d_{Q, \mu}$ in this way: we take $Q$ and we diagonalize it: $M^TQM=D$. We then set to zero at least $\mu$ of the eigenvalues getting a matrix $D'$ of smaller rank (in fact of corank bigger than $\mu$); finally we consider $Q'=MD'M^T\in Z^{\mu}$ (notice that for the generic $Q$ we get $n\choose \mu$ critical points of $d_{Q, \mu}$ on the smooth stratum of $Z^{\mu}$). 

Among all these critical points, minima on the smooth stratum $Z^\mu\backslash Z^{\mu+1}$ are obtained by setting to zero the $\mu$ smallest (in modulus) eigenvalues  $\sigma_{i_1}(Q), \ldots, \sigma_{i_{\mu}}(Q)$ Given one such minimum $X$, the distance $d_{V}(Q, X)$ is given by the square root of the Frobenius norm of $X-Q$, which equals:
$$\|Q-X\|=d_V(Q, X)=\sqrt{\sigma_{i_1}(Q)^2+\cdots+\sigma_{i_{\mu}}(Q)^2}.$$
Notice that the minimum on $Z^{\mu}$ of $d_{Q, \mu}$ is attained at a smooth point $X\in Z^{\mu}\backslash Z^{\mu+1}$.
\end{proof} \section{Random Matrix Theory}\label{sec:RMT}
Let $V$ as above denote respectively the vector space $\mathcal{M}(n, \RR)$ or $\textrm{Sym}(n, \RR)$. Using the Frobenius norm we can endow $V$ with a centered Gaussian probabilty distribution by letting for every open subset $U\subset V$:
\begin{equation}\label{gaussian}\PP\{Q\in U\}=C_V\int_{U}e^{-\|Q\|^2}dQ,\end{equation}
where $C_V=(\int_Ve^{-\|Q\|^2}dQ)^{-1}$ is the normalization constant (it can be computed explicitly, see for instance \cite{Fyodorov, Mehta}) and $dQ$ denotes the Lebesgue measure on the Euclidean space of the matrix entries.

Motivated by the last section, we consider the function:
$$p_{\mu}(\e)=\PP\{\sigma_{i_1}(Q)^2+ \cdots+ \sigma_{i_{\mu}}(Q)^2\leq \e^2\textrm{ for some $1\leq i_{1}, \ldots, i_{\mu}\leq n$}\}.$$
In the case $\mu=1$ the function $p_{1}(\e)$ is called the \emph{gap probability} and has been widely studied (see, for instance \cite{FoWy, Gaudin, JMMS}). Using the exclusion-inclusion principle we see that we can write:
\begin{equation}\label{gap}p_{\mu}(\e)=\sum_{k=\mu}^{n}(-1)^{k-\mu}{n\choose k}\underbrace{\PP\{\sigma_1(Q)^2+\cdots+\sigma_k(Q)^2\leq \e^2\}}_{g_k(\e)}\end{equation}
(the binomial coefficient comes from the fact that the distribution is invariant by orthogonal transformations; in particular it is invariant under permutations of the eigenvalues and we can just consider the first $k$ of them).

\begin{lemma}\label{red}
Let $c=\textrm{codim}_V(Z^{\mu})=\textrm{codim}_{S^{N-1}}(\Si^{\mu}).$ For every $k>\mu$ we have:
$$\lim_{\e\to 0}\frac{g_{k}(\e)}{\e^c}=0.$$\end{lemma}
\begin{proof}
It follows from the explicit computation for the eigenvalues density and we postpone it for later (after Proposition \ref{limreal} and Proposition \ref{limsym}).
\end{proof}

\begin{cor}\label{pi}Let $c=\textrm{codim}_V(Z^{\mu})=\textrm{codim}_{S^{N-1}}(\Si^{\mu})$.  We have:
$$\lim_{\e\to0}\frac{p_{\mu}(\e)}{\e^c}={n\choose \mu}\lim_{\e\to 0}\frac{g_\mu(\e)}{\e^c}.$$

\end{cor}
\begin{proof}
It is immediate after substituting the limit given in Lemma \ref{red} into the limit of \eqref{gap}.
\end{proof}

\begin{prop}\label{mainstep}
Let $c=\textrm{codim}_{V}(Z^\mu)$ and $g_{\mu}(\e)$ defined as above. Then:
\begin{equation}\label{formula}|\Si^{\mu}\cap V|=2^{c/2-1}\Gamma\left(\frac{c}{2}\right)\left|S^{N-1-c}\right|{n\choose \mu}\lim_{\e\to0}\frac{g_{\mu}(\e)}{\e^c}.\end{equation}
\end{prop}
\begin{proof}
We will use the result of Lemma \ref{vol}. We notice first that:
$$|\mathcal{U}_{V}(Z^{\mu}, \e)\cap S^{N-1}|=|S^{N-1}|\cdot \PP\{\sigma_{i_1}(Q)^2+ \cdots+ \sigma_{i_{\mu}}(Q)^2\leq \e\|Q\|^2\textrm{ for some $1\leq i_{1}, \ldots, i_{\mu}\leq n$}\}.$$
In fact Proposition \ref{cone} gives a geometric characterization of $\mathcal{U}_{V}(Z^{\mu}, \e)$ in terms of the singular values; we can rewrite this characterization as:
\begin{align*}\mathcal{U}_{V}(Z^\mu, \e)\cap S^{N-1}&=\{ \sigma_{i_1}(Q)^2+ \cdots+ \sigma_{i_{\mu}}(Q)^2\leq \e^2\textrm{ for some $1\leq i_{1}, \ldots, i_{\mu}\leq n$}\}\cap S^{N-1}\\
&=\underbrace{\{ \sigma_{i_1}(Q)^2+ \cdots+ \sigma_{i_{\mu}}(Q)^2\leq \e^2\|Q\|^2\textrm{ for some $1\leq i_{1}, \ldots, i_{\mu}\leq n$}\}}_{A_\mu(\e)}\cap S^{N-1}\end{align*}
(we have introduced a $\|Q\|^2$ factor in the second line, but this doesn't change our set since we are intersecting with the sphere $\{\|Q\|^2=1\}$).

The set $A_\mu(\e)$ as above is a homogeneous cone and since the probability distribution on $V$ is uniform on the sphere $S^{N-1}$ we have:
$$|\mathcal{U}_{V}(Z^{\mu}, \e)\cap S^{N-1}|=|A_{\mu}(\e)\cap S^{N-1}|=|S^{N-1}|\cdot \PP\{A_{\mu}(\e)\}.$$

Going back to equation \eqref{tubetobig} we obtain:
$$|\Si^\mu\cap V|=\frac{|S^{N-1}|}{|S^{c-1}|}\lim_{\e\to 0}\frac{\PP\{A_{\mu}(\e)\}}{\e^c}.$$
The probability of $A_{\mu}(\e)$ \emph{is not} $p_{\mu}(\e)$ (because of the $\|Q\|^2$ factor, the first is the probability of a ``cone'', the second of a ``cylinder''). The next Lemma \ref{lemma:cylcone} implies indeed we can study the limit of the one using the other by rescaling with the factor $2^{c/2}\Gamma(\frac{N}{2})\Gamma(\frac{N-c}{2})^{-1}$, obtaining:
$$|\Si^\mu\cap V|= \underbrace{\frac{|S^{N-1}|2^{c/2}\Gamma(\frac{N}{2})}{|S^{c-1}|\Gamma(\frac{N-c}{2})}}_{2^{c/2-1}\Gamma(c/2)|S^{N-1-c}|}\lim_{\e\to 0}\frac{p_{\mu}(\e)}{\e^c}.$$
Applying Corollary \ref{pi} concludes the proof.

\end{proof}

\begin{lemma}\label{lemma:cylcone}
Fix $N$ and $1\leq i_1, \ldots, i_{\mu}\leq n$; for $Q \in V$ call $\sigma(Q)^2=\sigma_{i_1}(Q)^2+ \cdots+ \sigma_{i_{\mu}}(Q)^2$. 
Then:
$$\lim_{\e \ra 0}\frac{ \PP\{\sigma(Q)^2\leq \e^2\|Q\|^2\} }{\e^c}= \frac{2^{c/2}\Gamma(\frac{N}{2})}{\Gamma(\frac{N-c}{2})} \lim_{\e\to 0}\frac{\PP\{\sigma(Q)^2\leq \e^2\}}{\e^c}.$$ 
\end{lemma}

\begin{proof}Let us call for simplicity of notation:
$$ f(\e)=\PP\{\sigma(Q)^2\leq \e^2\}\quad \textrm{and}\quad g(\e)=\PP\{\sigma(Q)^2\leq \e^2\|Q\|^2\}$$
First we establish the equation:
\begin{equation}\label{eq:f}f(\e)= \frac{|S^{N-1}|}{(2 \pi)^{N/2}} \int_0^\infty g(\e/r) r^{N-1} e^{-\frac{r^2}{2}} dr.
\end{equation}

Starting from the definition for $f$, we have:

\begin{align*}
f(\e) &= \frac{1}{(2 \pi)^{N/2}}  \int_0^\infty \int_{S^{N-1}} \chi_{\{ \sigma(Q) \geq \e \}}  r^{N-1} e^{-\frac{r^2}{2}} d\theta dr \\
&= \frac{1}{(2 \pi)^{N/2}}  \int_0^\infty \underbrace{\int_{S^{N-1}} \chi_{\{ \sigma(Q) \geq \e \}} d\theta}_{\text{Vol}(S^{N-1}) g(\e/r)} r^{N-1} e^{-\frac{r^2}{2}} dr \\
&=\frac{|S^{N-1}|}{(2 \pi)^{N/2}}  \int_0^\infty  g(\e/r) r^{N-1} e^{-\frac{r^2}{2}} dr.
\end{align*}

This proves (\ref{eq:f}). Since $g$ is differentiable at zero (up to the order $c$) so is $f$. Moreover for every $l\leq c$ we have:
$$ f^{(l)}(\e)=\frac{|S^{N-1}|}{(2 \pi)^{N/2}}  \int_0^\infty  g^{(l)}(\e/r) r^{N-1-l} e^{-\frac{r^2}{2}} dr.$$
For every $l\leq c$ the function $g^{(l)}(t)$ is continuous and has finite limits at zero and infinity: at zero simply because of Corollary  \ref{pi}; at infinity because $\lim_{t\to \infty}g(t)=1,$ hence $\lim_{t\to \infty}g'(t)=0$ and $\lim_{t\to \infty}g''(t)=0$ and so on for every derivative (if a differentiable function has a finite limit at infinity, then the derivative tends to zero). In particular every $g^{(l)}$ for $l\leq c$ is bounded by a constant $M_l$ and the family $g^{(l)}(\e/r) r^{N-1-l} e^{-\frac{r^2}{2}}$ is dominated by the integrable function $M_l r^{N-1-l} e^{-\frac{r^2}{2}} $. Hence for $l<c$ one can compute the limit:
\begin{align*}\lim_{\e\to 0}f^{(l)}(\e)&=\lim_{\e\to 0}\frac{|S^{N-1}|}{(2 \pi)^{N/2}}  \int_0^\infty  g^{(l)}(\e/r) r^{N-1-l} e^{-\frac{r^2}{2}} dr\\
&=\frac{|S^{N-1}|}{(2 \pi)^{N/2}}  \int_0^\infty  \lim_{\e\to 0}g^{(l)}(\e/r) r^{N-1-l} e^{-\frac{r^2}{2}} dr=0
\end{align*}
(in the last equality we have used Lemma \ref{red}).
In particular we can compute $\lim_{\e\to0}\frac{f(\e)}{\e^c}$ as (the chain of equalities is explained below):
\begin{align*}
\lim_{\e\to0}\frac{f(\e)}{\e^c}&=\lim_{\e\to 0}\frac{f^{(c)}(\e)}{c!}=\lim_{\e\to0}\frac{|S^{N-1}|}{(2 \pi)^{N/2}}  \int_0^\infty  \frac{g^{(c)}(\e/r)}{c!} r^{N-1-c} e^{-\frac{r^2}{2}} dr\\
&=\frac{|S^{N-1}|}{(2 \pi)^{N/2}}  \int_0^\infty  \lim_{\e\to0}\frac{g^{(c)}(\e/r)}{c!} r^{N-1-c} e^{-\frac{r^2}{2}} dr\\
&=\frac{|S^{N-1}|}{(2 \pi)^{N/2}}  \int_0^\infty  \lim_{\e\to0}\frac{g(\e)}{\e^{c}} r^{N-1-c} e^{-\frac{r^2}{2}} dr\\ 
&=\left( \lim_{\e \to0}\frac{g(\e)}{\e^c}\right)\frac{|S^{N-1}|}{(2 \pi)^{N/2}}  \int_0^\infty  r^{N-1-c} e^{-\frac{r^2}{2}} dr\\
&=\left( \lim_{\e \to0}\frac{g(\e)}{\e^c}\right)\frac{2^{-c/2}\Gamma\left(\frac{N-c}{2}\right)}{\Gamma\left(\frac{N}{2}\right)}.
\end{align*}
In the first equality we have used De l'Hopital's theorem; in the second equality we have taken derivative under the integral; in the third we we have used the dominated convergence theorem for the family $g^{(l)}(\e/r) r^{N-1-l} e^{-\frac{r^2}{2}}$ and exchanged the limit and the integral; the fourth equality is De l'Hopital's again (for $g$) and the last is simply the definition of the Gamma function.
In particular we obtain the desired relation:
$$\lim_{\e \ra 0}\frac{ \PP\{\sigma(Q)^2\leq \e^2\|Q\|^2\} }{\e^c}=  \lim_{\e \to0}\frac{g(\e)}{\e^c}=\frac{2^{c/2}\Gamma(\frac{N}{2})}{\Gamma(\frac{N-c}{2})} \lim_{\e\to 0}\frac{\PP\{\sigma(Q)^2\leq \e^2\}}{\e^c}.$$
\end{proof}
It remains to evaluate the limit $\lim_{\e\to 0}g_{\mu}(\e)/\e^c$ in equation \ref{formula}. We treat the two cases separately.
\subsection{Square matrices}Let us consider the singular values $\sigma_1, \ldots, \sigma_n$ of $Q\in \mathcal{M}(n, \RR)$; we do not assume they are ordered according to their magnitude. If $Q$ is sampled at random as above, the joint density of the singular values is given by:
$$\PP\{\sigma=(\sigma_1, \ldots, \sigma_n)\in B\}= C(n)\int_B e^{-\frac{1}{2}\|\sigma\|^2}\prod_{1\leq i<j\leq n}|\sigma_i^2-\sigma_j^2|d\sigma,\quad B\subseteq \RR^{n}_+$$
where $C(n)$ is the normalization constant and is given by:
$$C(n)^{-1}=\int_{\RR^{n}_+} e^{-\frac{1}{2}\|\sigma\|^2}\prod_{1\leq i<j\leq n}|\sigma_i^2-\sigma_j^2|d\sigma.$$
This normalization constant can be computed explicitly using Selberg's integral:
\begin{align*}C(n)^{-1}&=\int_{\RR^{n}_+} e^{-\frac{1}{2}\|\sigma\|^2}\prod_{1\leq i<j\leq n}|\sigma_i^2-\sigma_j^2|d\sigma\\
&=2^{-n}\int_{\RR^{n}} e^{-\frac{1}{2}\|\sigma\|^2}\prod_{1\leq i<j\leq n}|\sigma_i^2-\sigma_j^2|d\sigma\quad\textrm{(the function we're integrating is even)}\\
&=\frac{2^{n^2/2}}{\pi^{n/2}}\prod_{j=1}^n\Gamma(j/2)\Gamma(j/2+1)\quad\textrm{(\cite[eq. 17.6.6]{Mehta} with $\gamma=1/2$ and $\alpha=1/2$)}.\end{align*}
\begin{prop}\label{limreal}For the case of $n\times n$ real matrices:
$$\lim_{\e\to 0}\frac{g_{\mu}(\e)}{\e^c}=I_{\mu}\cdot \frac{\prod_{j=1}^{n-\mu}\Gamma(j/2+1)\Gamma(j/2+\mu)}{\prod_{j=1}^n \Gamma(j/2)\Gamma(j/2+1)},$$
where $I_{\mu}$ is a constant depending on $\mu$ only and is given by:
$$I_\mu=\frac{\pi^{\mu/2}}{2^{\mu^2/2}}\int_{B(0,1)\cap \RR^{\mu}_+}\prod_{1\leq i<j\leq \mu}|\sigma_i^2-\sigma_j^2|d\sigma_1\cdots d\sigma_{\mu}.$$

\begin{proof}
By definition $g_{\mu}(\e)=\PP\{\sigma_1^2+\cdots+\sigma_\mu^2\leq \e^2\}$ and using the joint density for $(\sigma_1, \ldots, \sigma_n)$ we can write:
$$g_{\mu}(\e)=C(n)\int_{B(0, \e)\cap \RR^{\mu}_+}\int_{\RR^{n-\mu}_+}e^{-1/2\|\sigma\|^2}\prod_{1\leq i<j\leq n}|\sigma_i^2-\sigma_j^2|d\sigma.$$
We split now the variables $(\sigma_1, \ldots, \sigma_n)$ into $\underline{\sigma}_{\mu}=(\sigma_1, \ldots, \sigma_\mu)$ and $\underline{\sigma}_{n-\mu}=(\sigma_{\mu+1}, \ldots, \sigma_n)$ and rewrite the integrand function as:
$$e^{-1/2\|\underline{\sigma}_{\mu}\|^2}\underbrace{\prod_{1\leq i<j\leq\mu}|\sigma_i^2-\sigma_j^2|}_{G(\underline\sigma_{\mu})}\underbrace{\prod_{1\leq i<\mu+1\leq j\leq n}|\sigma_i^2-\sigma_j^2|}_{H(\underline\sigma_{\mu}, \underline\sigma_{n-\mu})}\underbrace{e^{-1/2\|\underline{\sigma}_{n-\mu}\|^2}\prod_{\mu+1\leq i< j\leq n}|\sigma_i^2-\sigma_j^2|}_{F(\underline\sigma_{n-\mu})}.$$
We now change $\underline{\sigma}_\mu$ to polar coordinates $(r,\theta)\in(0, \e)\times (S^{\mu-1}\cap \RR^{\mu}_+)$, obtaining:
$$g_{\mu}(\e)=C(n)\int_{0}^{\e}\int_{S^{\mu-1}\cap \RR^{\mu}_+}\int_{\RR^{n-\mu}_+}F(\underline{\sigma}_{n-\mu})H(\underline{\sigma}_{\mu}(r, \theta), \underline{\sigma}_{n-\mu})e^{-r^2/2}G(\underline{\sigma}_{\mu}(r, \theta))r^{\mu-1}d\underline{\sigma}_{n-\mu}d\theta dr.$$
Performing now the change of variable $r=\e s$, we can rewrite the above integral as:
\begin{small}\begin{align*}g_{\mu}(\e)&=C(n)\int_{0}^{1}\int_{S^{\mu-1}\cap \RR^{\mu}_+}\int_{\RR^{n-\mu}_+}F(\underline{\sigma}_{n-\mu})H(\underline{\sigma}_{\mu}(\e s, \theta), \underline{\sigma}_{n-\mu})e^{-\e^2s^2/2}\underbrace{G(\underline{\sigma}_{\mu}(\e s, \theta))}_{\e^{\mu(\mu-1)} G(\underline{\sigma}_{\mu}( s, \theta))}(\e s)^{\mu-1}\e d\underline{\sigma}_{n-\mu}d\theta ds\\
&=\e^{\mu^2}C(n)\int_{0}^{1}\int_{S^{\mu-1}\cap \RR^{\mu}_+}\int_{\RR^{n-\mu}_+}F(\underline{\sigma}_{n-\mu})H(\underline{\sigma}_{\mu}(\e s, \theta), \underline{\sigma}_{n-\mu})e^{-\e^2s^2/2}G(\underline{\sigma}_{\mu}( s, \theta))s^{\mu-1} d\underline{\sigma}_{n-\mu}d\theta ds.
\end{align*}

\end{small}
In particular, since in this case $c=\mu^2$, we obtain:
\begin{small}
\begin{align*}\lim_{\e\to 0}\frac{g_{\mu}(\e)}{\e^c}&=\lim_{\e\to 0}C(n)\int_{0}^{1}\int_{S^{\mu-1}\cap \RR^{\mu}_+}\int_{\RR^{n-\mu}_+}F(\underline{\sigma}_{n-\mu})H(\underline{\sigma}_{\mu}(\e s, \theta), \underline{\sigma}_{n-\mu})e^{-\e^2s^2/2}G(\underline{\sigma}_{\mu}( s, \theta))s^{\mu-1} d\underline{\sigma}_{n-\mu}d\theta ds\\
&=C(n)\int_{0}^{1}\int_{S^{\mu-1}\cap \RR^{\mu}_+}\int_{\RR^{n-\mu}_+}\lim_{\e\to 0}F(\underline{\sigma}_{n-\mu})H(\underline{\sigma}_{\mu}(\e s, \theta), \underline{\sigma}_{n-\mu})e^{-\e^2s^2/2}G(\underline{\sigma}_{\mu}( s, \theta))s^{\mu-1} d\underline{\sigma}_{n-\mu}d\theta ds\\
&=C(n)\int_{0}^{1}\int_{S^{\mu-1}\cap \RR^{\mu}_+}\int_{\RR^{n-\mu}_+}F(\underline{\sigma}_{n-\mu})H(\underline{\sigma}_{\mu}(0, \theta), \underline{\sigma}_{n-\mu})G(\underline{\sigma}_{\mu}( s, \theta))s^{\mu-1} d\underline{\sigma}_{n-\mu}d\theta ds
\end{align*}
\end{small}
We have used the dominated convergence theorem in the second step: in fact the integrand is bounded by the integrable function $e^{-1/2\|\underline{\sigma}_{n-\mu}\|^2}p(s, \underline{\sigma}_{n-\mu})$, where $p$ is a polynomial (this function is integrable because $s\in [0,1];$ the exponential factor takes care of the polynomial part in $\underline{\sigma}_{n-\mu}$). To see that we can actually bound the integrand with such a function we proceed as follows: we bound $e^{-1/2\e^2s^2}$ by $1$; we bound each factor $|\e^2s^2\sigma_i(\theta)^2-\sigma_j^2|$ in $H$ by $|s^2+\sigma_j^2|$ and each factor $|s^2\sigma_i(\theta)^2-s^2\sigma_j(\theta)^2|$ in $G$ by $2 s^2$ (the functions $\sigma_i(\theta), i=1, \ldots, 
\mu$ are bounded by one)\footnote{For example if $n=3$ and $\mu=2$, the bounds for the integrand are:

\begin{align*}&e^{-\e^2s^2/2}e^{-1/2\sigma_3^2}s|s^2(\cos \theta)^2-s^2(\sin \theta)^2||\e^2s^2(\cos \theta)^2-\sigma_3^2||\e^2s^2(\sin \theta)^2-\sigma_3^2|\\
\leq & e^{-1/2\sigma_3^2}s |2s^2||s^2+\sigma_3^2||s^2+\sigma_3^2|=e^{-1/2 \sigma_3^2}2s^3(s^2+\sigma_3^2)^2\quad\quad(s>0)\end{align*}}.

Notice that:
$$H(0, \underline{\sigma}_{n-\mu})=\prod_{j=\mu+1}^{n}\sigma_i^{2\mu}.$$
In this way we have decoupled the variables of integration and, restoring euclidean coordinates for $\underline{\sigma}_\mu$, we can rewrite:\begin{tiny}
$$\lim_{\e\to 0}\frac{g_{\mu}(\e)}{\e^c}=C(n)\left(\int_{B(0,1)\cap \RR^{\mu}_+}\prod_{1\leq i<j\leq \mu}|\sigma_i^2-\sigma_j^2|d\underline{\sigma}_\mu\right)\cdot \left(\int_{\RR^{n-\mu}_+}e^{-\|\underline{\sigma}_{n-\mu}\|^2/2}\prod_{\mu+1\leq i<j\leq n}|\sigma_i^2-\sigma_j^2|\prod_{j=\mu+1}^{n}\sigma_i^{2\mu}d\underline{\sigma}_{n-\mu}\right). $$
\end{tiny}
The integral in the right factor can be computed again using Selberg's integral and equals:
\begin{align*}&\int_{\RR^{n-\mu}_+}e^{-\|\underline{\sigma}_{n-\mu}\|^2/2}\prod_{\mu+1\leq i<j\leq n}|\sigma_i^2-\sigma_j^2|\prod_{j=\mu+1}^{n}\sigma_i^{2\mu}d\underline{\sigma}_{n-\mu}\\
=&\,2^{n-\mu}\int_{\RR^{n-\mu}}e^{-\|\underline{\sigma}_{n-\mu}\|^2/2}\prod_{\mu+1\leq i<j\leq n}|\sigma_i^2-\sigma_j^2|\prod_{j=\mu+1}^{n}\sigma_i^{2\mu}d\underline{\sigma}_{n-\mu}\\
=&\, 2^{\frac{(n-\mu)(n+\mu)}{2}}\prod_{j=1}^{n-\mu}\frac{ \Gamma(j/2+1)\Gamma(j/2+\mu)}{\sqrt{\pi}}\quad \left(\textrm{\cite[eq. 17.6.6]{Mehta} with $\gamma=1/2$ and $\alpha=\frac{2\mu+1}{2}$}\right). \end{align*}
Plugging in the explicit expression of $C(n)$ and simplifying the constants, we obtain the result.\end{proof}
\end{prop}
We give now the proof of Lemma \ref{red} in the case $\mathcal{M}(n, \RR)$.
\begin{proof}Proceeding as in the above proof we have:
\begin{small}\begin{align*}g_{k}(\e)&=C(n)\int_{B(0, \e)\cap \RR^{k}_+}\int_{\RR^{n-k}_+}e^{-1/2\|\sigma\|^2}\prod_{1\leq i<j\leq n}|\sigma_i^2-\sigma_j^2|d\sigma\\
&=\e^{k^2}C(n)\int_{0}^{1}\int_{S^{k-1}\cap \RR^{k}_+}\int_{\RR^{n-k}_+}F(\underline{\sigma}_{n-l})H(\underline{\sigma}_{k}(\e s, \theta), \underline{\sigma}_{n-k})e^{-\e^2s^2/2}G(\underline{\sigma}_{l}( s, \theta))s^{k-1} d\underline{\sigma}_{n-k}d\theta ds.\end{align*}\end{small}
and $\lim_{\e}g_{k}(\e)/\e^{k^2}$ is a nonzero number. In particular if $k>\mu$ we have 
$$\lim_{\e\to 0}\frac{g_{l}(\e)}{\e^{\mu^2}}=\lim_{\e\to 0}\frac{g_{k}(\e)}{\e^{k^2}}\e^{k^2-\mu^2}=0.$$
\end{proof}
As a corollary, combining the  limit of Theorem \ref{volumereal} into Proposition \ref{mainstep}, we derive the following Theorem.
\begin{thm}\label{volumereal}
\begin{equation}\label{formula}\frac{|\Si^{\mu}\cap \mathcal{M}(n, \RR)|}{\left|S^{n^2-\mu^2-1}\right|}=2^{\mu^2/2-1}\Gamma\left(\frac{\mu^2}{2}\right)I_{\mu}\cdot {n\choose \mu}\frac{\prod_{j=1}^{n-\mu}\Gamma(j/2+1)\Gamma(j/2+\mu)}{\prod_{j=1}^n \Gamma(j/2)\Gamma(j/2+1)}.\end{equation}

\end{thm}

\subsection{Symmetric matrices}The case of symmetric matrices is a little more delicate and we present the result using Random Matrix Theory. The set $\textrm{Sym}(n, \RR)$ wit the Gaussian distribution \eqref{gaussian} is called the \emph{Gaussian Orthogonal Ensemble} and denoted by $\textrm{GOE}(n).$

The joint density of the (unordered) eigenvalues of a matrix in $\textrm{GOE}(n)$ is given by:
\begin{equation}\label{eigendensity}F_{1,n}(\lambda) = C_1(n) \exp\left(-\frac{1}{2} \sum_{j=1}^n \lambda_j^2 \right) \prod_{j,k \in [1,n]} |\lambda_k - \lambda_j |^{1/2},\end{equation}
where the normalization constant, computed using Selberg's integral, is given by (see \cite{Mehta, Fyodorov}):
\begin{equation}\label{densityconstant}C_1(n) = (2\pi)^{-n/2} \prod_{j=1}^{n} \frac{\Gamma(1+1/2)}{\Gamma(1+j/2)} .\end{equation}

\begin{prop}\label{limsym}For the case of $n\times n$ real, symmetric matrices:
$$\lim_{\e\to 0}\frac{g_{\mu}(\e)}{\e^c}=\left(\int_{B(0,1)\cap \RR^\mu}\prod_{i<j}|\lambda_i-\lambda_j|d\lambda_1\cdots d\lambda_{\mu}\right)\cdot \frac{C_1(n)}{C_1(n-\mu)}\EE_{Q\in \emph{\textrm{GOE}}(n-\mu)} |\det(Q)|^{\mu}$$

\end{prop}
\begin{proof}
The proof is very similar to the previous one, except that we use the explicit density for the unordered eigenvalues $\lambda_1, \ldots, \lambda_n$ of $Q$, in such a way that:
$$g_\mu(\epsilon)=\PP\{\lambda_1(Q)^2+\cdots+\lambda_\mu(Q)^2\leq \epsilon^2\}$$
(i.e. we use the eigenvalues description instead of the least singular values one). Thus we obtain:

$$g_{\mu}(\e)=C_1(n)\int_{B(0, \e)\cap \RR^{\mu}}\int_{\RR^{n-\mu}}e^{-1/2\|\lambda\|^2}\prod_{1\leq i\leq j\leq n}|\lambda_i-\lambda_j|^{1/2}d\lambda.$$
We split now as above the variables $(\lambda_1, \ldots, \lambda_n)$ into $\underline{\lambda}_{\mu}=(\lambda_1, \ldots, \lambda_\mu)$ and $\underline{\lambda}_{n-\mu}=(\lambda_{\mu+1}, \ldots, \lambda_n)$ and we rewrite the integrand function as:
$$e^{-1/2\|\underline{\lambda}_{\mu}\|^2}\underbrace{\prod_{1\leq i\leq j\leq\mu}|\lambda_i-\lambda_j|^{1/2}}_{G(\underline\lambda_{\mu})}\underbrace{\prod_{1\leq i\leq \mu+1\leq j\leq n}|\lambda_i-\lambda_j|^{1/2}}_{H(\underline\lambda_{\mu}, \underline\lambda_{n-\mu})}\underbrace{e^{-1/2\|\underline{\lambda}_{n-\mu}\|^2}\prod_{\mu+1\leq i\leq j\leq n}|\lambda_i-\lambda_j|^{1/2}}_{F(\underline\lambda_{n-\mu})}.$$
We now change $\underline{\lambda}_\mu$ to polar coordinates $(r,\theta)\in(0, \e)\times S^{\mu-1}$, obtaining:
$$g_{\mu}(\e)=C_1(n)\int_{0}^{\e}\int_{S^{\mu-1}}\int_{\RR^{n-\mu}}F(\underline{\lambda}_{n-\mu})H(\underline{\lambda}_{\mu}(r, \theta), \underline{\lambda}_{n-\mu})e^{-r^2/2}G(\underline{\lambda}_{\mu}(r, \theta))r^{\mu-1}d\underline{\lambda}_{n-\mu}d\theta dr.$$
Performing now the change of variable $r=\e s$, we can rewrite the above integral as:
\begin{small}\begin{align*}g_{\mu}(\e)&=C_1(n)\int_{0}^{1}\int_{S^{\mu-1}}\int_{\RR^{n-\mu}}F(\underline{\lambda}_{n-\mu})H(\underline{\lambda}_{\mu}(\e s, \theta), \underline{\lambda}_{n-\mu})e^{-\e^2s^2/2}\underbrace{G(\underline{\lambda}_{\mu}(\e s, \theta))}_{\e^{{\mu \choose 2}} G(\underline{\lambda}_{\mu}( s, \theta))}(\e s)^{\mu-1}\e d\underline{\lambda}_{n-\mu}d\theta ds\\
&=\e^{\mu+{\mu\choose 2}}C_1(n)\int_{0}^{1}\int_{S^{\mu-1}}\int_{\RR^{n-\mu}}F(\underline{\lambda}_{n-\mu})H(\underline{\lambda}_{\mu}(\e s, \theta), \underline{\lambda}_{n-\mu})e^{-\e^2s^2/2}G(\underline{\lambda}_{\mu}( s, \theta))s^{\mu-1} d\underline{\lambda}_{n-\mu}d\theta ds.
\end{align*}

\end{small}
In particular, since in this case $c=\mu(\mu+1)/2$, we obtain:
\begin{small}
\begin{align*}\lim_{\e\to 0}\frac{g_{\mu}(\e)}{\e^c}&=\lim_{\e\to 0}C_1(n)\int_{0}^{1}\int_{S^{\mu-1}}\int_{\RR^{n-\mu}}F(\underline{\lambda}_{n-\mu})H(\underline{\lambda}_{\mu}(\e s, \theta), \underline{\lambda}_{n-\mu})e^{-\e^2s^2/2}G(\underline{\lambda}_{\mu}( s, \theta))s^{\mu-1} d\underline{\lambda}_{n-\mu}d\theta ds\\
&=C_1(n)\int_{0}^{1}\int_{S^{\mu-1}}\int_{\RR^{n-\mu}}\lim_{\e\to 0}F(\underline{\lambda}_{n-\mu})H(\underline{\lambda}_{\mu}(\e s, \theta), \underline{\lambda}_{n-\mu})e^{-\e^2s^2/2}G(\underline{\lambda}_{\mu}( s, \theta))s^{\mu-1} d\underline{\lambda}_{n-\mu}d\theta ds\\
&=C_1(n)\int_{0}^{1}\int_{S^{\mu-1}}\int_{\RR^{n-\mu}}F(\underline{\lambda}_{n-\mu})H(\underline{\lambda}_{\mu}(0, \theta), \underline{\lambda}_{n-\mu})G(\underline{\lambda}_{\mu}( s, \theta))s^{\mu-1} d\underline{\lambda}_{n-\mu}d\theta ds
\end{align*}
\end{small}
(we have used the dominated convergence theorem in the second step). Notice that:
$$H(0, \underline{\lambda}_{n-\mu})=\prod_{j=\mu+1}^{n}|\lambda_i|^{\mu}.$$
In this way we have decoupled the variables of integration and, restoring euclidean coordinates for $\underline{\lambda}_\mu$, we can rewrite:\begin{tiny}
$$\lim_{\e\to 0}\frac{g_{\mu}(\e)}{\e^c}=C_1(n)\left(\int_{B(0,1)\cap \RR^{\mu}}\prod_{1\leq i< j\leq \mu}|\lambda_i-\lambda_j|d\underline{\lambda}_\mu\right)\cdot \left(\int_{\RR^{n-\mu}}e^{-\|\underline{\lambda}_{n-\mu}\|^2/2}\prod_{\mu+1\leq i <j\leq n}|\lambda_i-\lambda_j|\prod_{j=\mu+1}^{n}|\lambda_i|^{\mu}d\underline{\lambda}_{n-\mu}\right) $$
\end{tiny}
and dividing and multiplying by $C_1(n-\mu)$ concludes the proof.
\end{proof}

\begin{remark}Arguing as immediately after Theorem \ref{limreal} we get the proof of Lemma \ref{red} for symmetric matrices.
\end{remark}

\begin{remark}\label{remGOE}Notice that the integral $\int_{B(0,1)\cap \RR^\mu}\prod_{i<j}|\lambda_i-\lambda_j|d\lambda_1\cdots d\lambda_{\mu}$ depends only on $\mu$, although its exact evaluation as a function of $\mu$ is not trivial.

As for  $\frac{C_1(n)}{C_1(n-\mu)}\EE_{Q\in \textrm{GOE}(n-\mu)} |\det(Q)|^{\mu}, $ explicit computations are subtle and depend on the parity of $n$ and $\mu$. 

If $n-\mu=2m+1$, then by \cite[Eq. 26.5.2]{Mehta}:
$$\begin{small}\EE_{Q\in \textrm{GOE}(2m+1)} |\det(Q)|^{\mu}=\Gamma\left(\frac{\mu+1}{2}\right)2^{\frac{\mu+1}{2}}(2\pi)^{-1/2}\prod_{i=0}^{m-1}\frac{\Gamma(i+\mu+3/2)}{\Gamma(i+3/2)}\end{small}\quad (n-\mu=2m+1).$$
The case $n-\mu$ is even is discussed in \cite[26.6]{Mehta}, but is more complicated.

\end{remark}


\begin{remark}
A similar computation can be performed for the set of hermitian and quaternionic hermitian matrices. They correspond to the classical Gaussian $\beta$-ensembles ($\beta=1$ is the GOE, $\beta=2$ the GUE and $\beta=4$ the GSE). The result is analogue, except that the codimension of matrices with $\mu$-dimensional kernel is $\beta{\mu\choose2}$, and we easily obtain:
$$\lim_{\e\to 0}\frac{g_{\mu}(\e)}{\e^c}=\left(\int_{B(0,1)\cap \RR^\mu}\prod_{i<j}|\lambda_i-\lambda_j|^\beta d\lambda_1\cdots d\lambda_{\mu}\right)\cdot \frac{C_\beta(n)}{C_\beta(n-\mu)}\EE_{Q\in G_\beta(n-\mu)} |\det(Q)|^{\mu \beta}$$

\end{remark}
As a corollary, combining the limit of Theorem \ref{limsym} into Proposition \ref{mainstep} and the explicit expression for the expectation of $|\det (Q)|^{\mu}$ given in Remark \ref{remGOE},  we derive the following Theorem.
\begin{thm}\label{volsym} \begin{align*}
\frac{|\Si^{\mu}\cap \emph{\textrm{Sym}}(n, \RR)|}{\left|S^{\frac{n(n+1)}{2}-\frac{\mu(\mu+1)}{2}-1}\right|}&=2^{\frac{\mu(\mu+1)}{4}-1}\Gamma\left(\frac{\mu(\mu+1)}{4}\right) I_{1, \mu}\cdot {n\choose\mu}\cdot \frac{C_1(n)}{C_1(n-\mu)}\cdot \EE_{Q\in \emph{\textrm{GOE}}(n-\mu)}|\det(Q)|^{\mu}\\
&=\frac{2^{\frac{\mu(\mu+1)}{4}-1}\Gamma\left(\frac{\mu(\mu+1)}{4}\right)\Gamma\left(\frac{\mu+1}{2}\right)}{\sqrt{\pi}}I_{1, \mu}\cdot {n\choose \mu}\frac{\prod_{i=0}^{m-1}\frac{\Gamma(i+\mu+3/2)}{\Gamma(i+3/2)}}{\prod_{k=n-\mu+1}^{n}\Gamma(1+k/2)}
\end{align*}
where $$n-\mu=2m+1\quad \textrm{and}\quad I_{1, \mu}=2^{-\mu}\int_{B(0,1)\cap \RR^\mu}\prod_{i<j}|\lambda_i-\lambda_j|d\lambda_1\cdots d\lambda_{\mu}$$

\end{thm}

\section{Asymptotic analysis}\label{sec:asy}

In this section we perform the asymptotic analysis of the previous results. 

\begin{thm}\label{thm:realas}$$\frac{|\Si^{\mu}\cap \mathcal{M}(n, \RR)|}{\left|S^{n^2-\mu^2-1}\right|}=\Theta\left( n^{\frac{\mu^2}{2}}\right).$$

\end{thm}

\begin{proof}
Using equation \eqref{formula} we have:
$$\frac{|\Si^{\mu}\cap \mathcal{M}(n, \RR)|}{\left|S^{N-1-\mu^2}\right|}\sim c_1(\mu)\cdot n^{\mu}\cdot \underbrace{ \frac{\prod_{j=1}^{n-\mu}\Gamma(j/2+1)\Gamma(j/2+\mu)}{\prod_{j=1}^n \Gamma(j/2)\Gamma(j/2+1)}}_{a(n, \mu)}$$
where $c_1(\mu)$ is a constant that depends on $\mu$ only. For the asymptotic of $a(n, \mu)$ we proceed as follows. First simplifying the factors we obtain:
\begin{equation}\label{a1}a(n ,\mu)=\frac{\prod_{j=1}^{n-\mu}\Gamma(j/2+\mu)}{\left(\prod_{j=n-\mu+1}^n\Gamma(j/2+1)\right)\cdot\left(\prod_{j=1}^n \Gamma(j/2)\right)}.\end{equation}
For the numerator we iterate the multiplicative formula $z\Gamma(z)=\Gamma(z+1)$ which gives:
\begin{align*}\Gamma(j/2+\mu)&=(j/2+\mu-1)\cdot(j/2+\mu-2)\cdots (j/2+1)\cdot j/2\cdot \Gamma(j/2)\\
&=2^{-\mu}\underbrace{(j+2\mu-2)\cdots (j+2\mu-4)\cdots (j+2)\cdot j }_{b(j, \mu)}\cdot \Gamma(j/2)\end{align*}
This allows to rewrite:
\begin{equation}\label{a2}a(n, \mu)=c_2(\mu)\cdot 2^{-n \mu}\frac{\prod_{j=1}^{n-\mu}b(j, \mu)}{\prod_{j=n-\mu+1}^n\Gamma(j/2+1) \Gamma(j/2)}.\end{equation}
The term $\prod_{j=1}^{n-\mu}b(j, \mu)$ is the product of all the elements in the following table:
$$
\begin{array}{|c|c|c|c|c|}
\hline

\hline
 1&3&\cdots&2\mu-3&2\mu-1\\
\hline
\vdots&\vdots&&\vdots&\vdots\\

\hline
j&j+2&\cdots&j+2\mu&j+2\mu-2\\
\hline
\vdots&\vdots&&\vdots&\vdots\\
\hline
n-\mu&n-\mu+2&\cdots&n+\mu-4&n+\mu-2\\
\hline
\end{array}
$$
($b(j, \mu)$ is the product of all the elements on the $j$-th row). Performing the multiplication columnwise first we obtain that the product of all the elements in the $k$-th column equals $\Gamma(n-\mu+2k-1)/\Gamma(2k-1)$ and since the number of columns is $\mu$ (it doesn't depend on $n$), then:
\begin{equation}\label{prod}\prod_{j=1}^{n-\mu}b(j,\mu)=c_3(\mu)\cdot \prod_{k=1}^{\mu}\Gamma(n-\mu+2k-1).\end{equation}
Let's consider now the term $\prod_{j=n-\mu+1}^n\Gamma(j/2+1) \Gamma(j/2)$ in \eqref{a2}. Using the doubling identity $\Gamma(z)\Gamma(z+1/2)=2^{1-2z}\sqrt{\pi}\Gamma(z)$ with $z=j+1$, we can rewrite each term in this product as:
\begin{align*}\Gamma(j/2+1)\Gamma(j/2)&=\Gamma(j/2+1)\Gamma(j/2+1/2)\frac{\Gamma(j/2)}{\Gamma(j/2+1/2)}\\
&=\sqrt{\pi}\Gamma(j+1)2^{-j-1}\frac{\Gamma(j/2)}{\Gamma(j/2+1/2)}.\end{align*}
In particular we obtain:
$$\prod_{j=n-\mu+1}^n\Gamma(j/2+1) \Gamma(j/2)=c_4(\mu)\cdot 2^{-n\mu}\prod_{j=n-\mu+1}^{n}\Gamma(j+1)\frac{\Gamma(j/2)}{\Gamma(j/2+1/2)}.$$
Moreover, since in the above product $n-j\leq \mu$ and $\mu$ is fixed, we can use the asymptotic formula $\Gamma(z+a)/\Gamma(z+b)\sim z^{a-b}$ for $z=j/2, a=0, b=1/2$ and obtain:
$$\frac{\Gamma(j/2)}{\Gamma(j/2+1/2)}\sim \left(j/2\right) ^{-1/2}$$
which substituted into the above formula gives:
\begin{align*}\prod_{j=n-\mu+1}^n\Gamma(j/2+1) \Gamma(j/2)&=c_4(\mu)\cdot 2^{-n\mu}\prod_{j=n-\mu+1}^{n}\Gamma(j+1) \left(j/2\right) ^{-1/2}\\
&\sim c_5(\mu)\cdot 2^{-n\mu}n^{-\mu/2}\prod_{j=n-\mu+1}^{n}\Gamma(j+1)\\
&= c_5(\mu)\cdot 2^{-n\mu}n^{-\mu/2}\prod_{k=1}^{\mu}\Gamma(n-\mu+k-1).
\end{align*}
Cxombining this asymptotic and \eqref{prod} into \eqref{a2} we obtain:
$$a(n, \mu)\sim c_6(\mu)\cdot n^{\mu/2}\prod_{k=1}^{\mu}\frac{\Gamma(n-\mu+2k-1)}{\Gamma(n-\mu+k+1)}\sim c_7(\mu)\cdot n^{\mu/2}\prod_{k=1}^{\mu}n^{k-2}$$
where we have used again the asymptotic formula $\Gamma(z+a)/\Gamma(z+b)\sim z^{a-b}$ with $z=n-\mu, a=2k+1, b=k+1$.
In this way we finally obtain:
\begin{align*}\frac{|\Si^{\mu}\cap \mathcal{M}(n, \RR)|}{\left|S^{n^2-\mu^2-1}\right|}&\sim c_1(\mu) \cdot n^{\mu} a(n, \mu)\\
&\sim c_8(\mu)\cdot  n^{\mu+\mu/2}\prod_{k=1}^{\mu}n^{k-2}\sim c_8(\mu)\cdot  n^{\mu^2/2}.
\end{align*}

\end{proof}

We perform now the asymptotic analysis of Theorem \ref{volsym} (for the case $n-\mu=2m+1$).
\begin{thm}\label{thm:symas}
$$\frac{|\Si^{\mu}\cap \emph{\textrm{Sym}}(n, \RR)|}{\left|S^{\frac{n(n+1)}{2}-\frac{\mu(\mu+1)}{2}-1}\right|}=\Theta\left(n^{\frac{\mu(\mu+1)}{4}}\right).$$
\end{thm}

\begin{proof}
By Corollary \ref{volsym} we can write:
$$\frac{|\Si^{\mu}\cap \textrm{Sym}(n, \RR)|}{\left|S^{\frac{n(n+1)}{2}-\frac{\mu(\mu+1)}{2}-1}\right|}\sim c'_1(\mu)\cdot n^{\mu}\cdot \frac{\prod_{i=0}^{m-1}\frac{\Gamma(i+\mu+3/2)}{\Gamma(i+3/2)}}{\prod_{k=n-\mu+1}^{n}\Gamma(1+k/2)}.$$
We start by analyizing the factor $\Gamma(i+\mu+3/2)$:
\begin{align*}\Gamma(i+\mu+3/2)&=(i+\mu+3/2-1)\Gamma(i+\mu+3/2-1)\\
&=(i+\mu+3/2-1)(i+\mu+3/2-2)\cdots (i+3/2+1)(i+3/2)\Gamma(i+3/2)\\
&=2^\mu(2i+2\mu+1)(2i+2\mu-1)\cdots(2i+5)(2i+3)\Gamma(i+3/2)
\end{align*}
In particular we obtain:
\begin{equation}\label{p1}\prod_{i=0}^{m-1}\frac{\Gamma(i+\mu+3/2)}{\Gamma(i+3/2)}=2^{-m\mu}\underbrace{\prod_{i=0}^{m-1}(2i+2\mu+1)(2i+2\mu-1)\cdots(2i+5)(2i+3)}_{a'(m, \mu)}.\end{equation}
The term $a'(m, \mu)$ is the product of all the elements in the following table:
$$\begin{array}{|c|c|c|c|c|}
\hline

\hline
 3&5&\cdots&2\mu-1&2\mu+1\\
\hline
\vdots&\vdots&&\vdots&\vdots\\

\hline
2i+2\mu+1&2i+2\mu+3&\cdots&2i+1&2i+3\\
\hline
\vdots&\vdots&&\vdots&\vdots\\
\hline
2m+1&2m+3&\cdots&2m+2\mu-3&2m+2\mu-1\\
\hline
\end{array}
$$
The product of all the elements in the $k$-th column equals $(2m+2k-1)!!/(2k-1)!!$ and since the number of columns is $\mu$ (which is independent on $n$) we can rewrite:
$$a'(m,\mu)=c_2'(\mu)\cdot\prod_{k=1}^{\mu}(2m+2k-1)!!=c_2'(\mu)\cdot\prod_{k=1}^{\mu}\Gamma(m+k+1/2)2^{m+k}\pi^{-1/2}$$
(we have used the identity $\Gamma\left(\frac{z+1}{2}\right)=(z-1)!!\sqrt{\pi}2^{-z/2}$ for $z=2m+2k$). Substituting this into \eqref{p1} we obtain:
$$\prod_{i=0}^{m-1}\frac{\Gamma(i+\mu+3/2)}{\Gamma(i+3/2)}=c_3'(\mu)\cdot \prod_{k=1}^{\mu}\Gamma(m+k+1/2)$$
Recalling that we have assumed $m=n/2-\mu/2-1/2$ and changing the index of multiplication:
$$\prod_{k=n-\mu+1}^{n}\Gamma(1+k/2)=\prod_{k=1}^{\mu}\Gamma(1+n/2-\mu/2+k/2)$$
we obtain:
\begin{equation}\label{p2}
\frac{|\Si^{\mu}\cap \textrm{Sym}(n, \RR)|}{\left|S^{\frac{n(n+1)}{2}-\frac{\mu(\mu+1)}{2}-1}\right|}=c_4'(\mu)\cdot n^{\mu}\cdot\prod_{k=1}^{\mu}\frac{\Gamma\left(\frac{n-\mu}{2}+k\right)}{\Gamma\left(\frac{n-\mu}{2}+\frac{k}{2}+1\right)}
\end{equation}
As $n\to \infty$, since $\mu$ is fixed, we can use the asymptotic:
$$\frac{\Gamma\left(\frac{n-\mu}{2}+k\right)}{\Gamma\left(\frac{n-\mu}{2}+\frac{k}{2}+1\right)}\sim\left(\frac{n-\mu}{2}\right)^{\frac{k}{2}-1}\sim c_5'(\mu)\cdot n^{\frac{k}{2}-1}$$
which substituted into \eqref{p2} gives:
$$\frac{|\Si^{\mu}\cap \textrm{Sym}(n, \RR)|}{\left|S^{\frac{n(n+1)}{2}-\frac{\mu(\mu+1)}{2}-1}\right|}\sim c_6'(\mu)\cdot n^{\mu}\prod_{k=1}^{\mu}n^{k/2-1}= c_6'(\mu)\cdot n^{\frac{\mu(\mu+1)}{4}}.$$
\end{proof}

\end{document}